\documentclass[12pt,reqno,a4paper]{amsart}

\usepackage{amssymb}
\usepackage{graphicx}
\usepackage[latin1]{inputenc}

\numberwithin{equation}{section}

\def\ZZ{{\mathbb Z}}

\def\Hdepth{\operatorname{Hdepth}}
\def\depth{\operatorname{depth}}

\def\coef#1{\left\langle#1\right\rangle}
\def\cl#1{\lceil#1\rceil}
\def\Ga{\Gamma}
\def\ep{\varepsilon}

\def\mm{{\mathfrak m}}

\newtheorem{lemma}{Lemma}[section]

\newtheorem{theorem}[lemma]{Theorem}
\newtheorem{proposition}[lemma]{Proposition}

\theoremstyle{definition}

\newtheorem{remark}[lemma]{Remark}

\title{Hilbert depth of powers of the maximal ideal}

\author{Winfried Bruns}
\address{Universit\"at Osnabr\"uck, Institut f\"ur Mathematik, 49069 Osnabr\"uck, Germany}
\email{wbruns@uos.de}

\author{Christian Krattenthaler$^{\dagger}$}
\address{Fakult\"at f\"ur Mathematik, Universit\"at Wien,
Nordbergstrasze~15, A-1090 Vienna, Austria.
WWW: \tt http://www.mat.univie.ac.at/\lower0.5ex\hbox{\~{}}kratt.}

\author{Jan Uliczka}
\address{Universit\"at Osnabr\"uck, Institut f\"ur Mathematik, 49069 Osnabr\"uck, Germany}
\email{Jan.Uliczka@uos.de}

\begin{document}

\thanks{$^\dagger$Research partially supported by the Austrian
Science Foundation FWF, grants Z130-N13 and S9607-N13,
the latter in the framework of the National Research Network
``Analytic Combinatorics and Probabilistic Number Theory"}

\begin{abstract}
The Hilbert depth of a module $M$ is the maximum depth that occurs
among all modules with the same Hilbert function as $M$. In this
note we compute the Hilbert depths of the powers of the irrelevant
maximal ideal in a standard graded polynomial ring.
\end{abstract}

 \maketitle

 \section{Introduction}

 In \cite{BKU} and \cite{Ul} the authors have investigated the
relationship between Hilbert series and depths of graded modules
over standard graded and multigraded polynomial rings. In this paper
we will consider only the standard graded case, i.e., finitely
generated graded modules over polynomial rings $R=K[X_1,\dots,X_n]$
for which $K$ is a field and $\deg X_i=1$ for $i=1,\dots,n$.

We refer the reader to \cite{BH} for the basic theory of Hilbert
functions and series. Let us just recall that the \emph{Hilbert
function} of a graded $R$-module $M=\bigoplus_{k\in \ZZ} M_k$ is
given by
$$
H(M,k)=\dim_K M_k,\qquad k\in \ZZ,
$$
and that the Laurent series
$$
H_M(T)=\sum_{k\in\ZZ} H(M,k)\,T^k
$$
is called the \emph{Hilbert series} of $M$. The Hilbert series is
the Laurent expansion at $0$ of a rational function as in
\eqref{rational} with a Laurent polynomial in the numerator.

Let us say that a Laurent series $\sum_{k\in\ZZ}a_kT^k$ is
\emph{positive} if $a_k\ge0$ for all $k$. Hilbert series are
positive by definition, and it is not surprising that positivity is
the central condition in the following theorem that summarizes the
results of \cite{Ul}. It describes the maximum depth that a graded
module with given Hilbert series can have.

\begin{theorem}\label{standard}
Let $R=K[X_1,\dots,X_n]$ as above, and let $M\neq0$ be a finitely
generated graded $R$-module with Hilbert series
\begin{equation}\label{rational}
H_M(T)=\frac{Q_M(T)}{(1-T)^n},\qquad Q_M(T)\in\ZZ[T,T^{-1}].
\end{equation}
Then the following numbers coincide:
\begin{enumerate}
\item $\max\{\depth N: H_M(T)=H_N(T)\}$,
\item the maximum $d$ such that $H_M(T)$ can be written as
\begin{equation}\label{partial}
H_M(T)=\sum_{e=d}^n \frac{Q_e(T)}{(1-T)^e},\qquad Q_e(T)\in\ZZ[T,T^{-1}],
\end{equation}
with positive Laurent polynomials $Q_e(T)$,
\item $\max\{p: (1-T)^p H_M(T)\text{ positive}\}$,
\item $n-\min\{q: Q_M(T)/(1-T)^q\text{ positive}\}$.
\end{enumerate}
\end{theorem}

The crucial point of the proof of Theorem~\ref{standard}  is to show
that every positive Laurent series that can be written in the form
\eqref{rational} has a representation of type \eqref{partial} (with
$d\ge 0$).

In \cite{Ul}, Theorem~\ref{standard}(1) is used to define the
\emph{Hilbert depth} $\Hdepth M$ of $M$, whereas \cite{BKU} bases
the definition of Hilbert depth on (2). In view of the theorem, this
difference is irrelevant in the standard graded case, but in the
multigraded case the equivalence of \ref{standard}(1) and a suitable
generalization of (2) is widely open, and can be considered as a
Hilbert function variant of the Stanley conjecture (see \cite{BKU}
for this connection). (Note that Theorem~\ref{standard}(3) and (4)
cannot be transferred easily to the multigraded situation.)

The Hilbert depth of the maximal ideal $\mm=(X_1,\dots,X_n)$ is
known:
$$
\Hdepth \mm=\left\lceil \frac n2\right\rceil.
$$
This was observed in \cite{Ul} and will be proved again below. (In
fact, by a theorem of Bir\'o et al.\ \cite{BHKTY}, the (multigraded)
Stanley depth of $\mm$ is given by $\lceil n/2\rceil$.)

In general, Hilbert depth is hard to compute since one almost
inevitably encounters alternating expressions whose nonnegativity is
to be decided. Not even for all syzygy modules of $\mm$ Hilbert
depth is known precisely; see \cite{BKU}. Nevertheless, the main
result of our paper shows that the Hilbert depths of the powers of
$\mm$ can be determined exactly.

\begin{theorem}\label{powers}
For all $n$ and $s$ one has
$$
\Hdepth \mm^s=\left\lceil \frac n{s+1}\right\rceil.
$$
\end{theorem}

That $\lceil n/(s+1)\rceil$ is an upper bound is seen easily. The
Hilbert series of $\mm^s$ is
$$
\binom{n+s-1}{s}T^s+\binom{n+s}{s+1}T^{s+1}+\cdots
$$
Thus the coefficient of $T^{s+1}$ in $(1-T)^pH_{\mm^s}(T)$ is
$$
\binom{n+s}{s+1}-p\binom{n+s-1}{s},
$$
and the difference is negative unless
$p\le\lfloor(n+s)/(s+1)\rfloor=\lceil n/(s+1)\rceil$. That the
condition $p\le\lceil n/(s+1)\rceil$ is sufficient for the
positivity of $(1-T)^pH_{\mm^s}(T)$ will be shown in the
remainder of this paper.

\medskip
The first step in the proof is the computation of
$(1-T)^rH_{\mm^s}(T)$ since, in view of Theorem~\ref{standard}(3),
we want to find the maximum $r$ for which this series is positive.

\begin{proposition} \label{prop:0}
For any integer $0 < r < n$, we have
\begin{multline} \label{eq:crit}
(1-T)^r H_{\mm^s}(T) = \binom{n+s-1}{s} T^s \\
+ \sum_{k=s+1}^{r+s-1} \left[ \binom{n+k-1-r}{k}
+ (-1)^{k-1} \sum_{j=0}^{s-1} (-1)^j \binom{r}{k-j} \binom{n+j-1}{j} \right] T^k \\
+ \sum_{k=r+s}^{\infty}  \binom{n+k-1-r}{k}T^k.
\end{multline}
\end{proposition}

This is easily proved by induction on $r$ or by the binomial expansion
of $(1-T)^r$.

The critical term in \eqref{eq:crit} is the one in the middle row.
In the next section we will find an alternative expression for it.
The positivity of this expression for $r\le \lceil n/(s+1)\rceil$
will be stated in Proposition~\ref{prop:1}. Its proof is the subject
of Section~3.

\bigskip\noindent
{\bf Acknowledgement.}
We are indebted to Jiayuan Lin for pointing out a gap in the proof of
Lemma~\ref{lem:5} in the first version of this article, 
for providing an argument that fixes this gap, and
for giving us the permission to reproduce this argument here.

\section{Binomial identities}

One of the ingredients in the proof of Theorem~\ref{powers} via
Proposition~\ref{prop:1} in the next section are two alternative
expressions for the sum over $j$ in the critical term in
\eqref{eq:crit}. We provide two proofs: a direct one using
well-known identities for binomial and hypergeometric series, and a
--- perhaps more illuminating --- algebraic one which shows that the
three expressions give the Hilbert function of a certain module.

\begin{lemma} \label{lem:4}
For all positive integers $n,s,r,k$, we have
\begin{align}\notag
\sum _{j=s} ^{k}(-1)^{k-j}&\binom {n+j-1}j\binom r{k-j}\\
\label{eq:8}
&=
\binom {n+k-r-1}k+(-1)^{k-1}
\sum _{j=0} ^{s-1}(-1)^j\binom r{k-j}\binom {n+j-1}j\\
\label{eq:9}
&=
\binom {n+k-r-1}k
+(-1)^{k+s}\sum _{t=1}^{r}\binom {r-t}{k-s}\binom {n-t+s-1}{s-1}.
\end{align}
\end{lemma}

\begin{proof}
We start with the direct proof.
Using the short
notation $\coef{T^k}f(T)$ for the coefficient of $T^k$ in the power
series $f(z)$, we have
\begin{align*}
\sum _{j=s} ^{k}(-1)^{k-j}&\binom {n+j-1}j\binom r{k-j}=
\coef{T^k}(1-T)^r
\sum _{j=s} ^{\infty}\binom {n+j-1}jT^j\\
&=
\coef{T^k}(1-T)^r\left((1-T)^{-n}-
\sum _{j=0} ^{s-1}\binom {n+j-1}jT^j\right)\\
&=
\coef{T^k}(1-T)^{r-n}-
\coef{T^k}(1-T)^{r}
\sum _{j=0} ^{s-1}\binom {n+j-1}jT^j\\
&=
\binom {n+k-r-1}k-
\sum _{j=0} ^{s-1}(-1)^{k-j}\binom r{k-j}\binom {n+j-1}j,
\end{align*}
which proves \eqref{eq:8}.

In order to see the equality between \eqref{eq:8} and
\eqref{eq:9}, we have to prove
\begin{equation} \label{eq:10}
\sum _{j=0} ^{s-1}(-1)^j\binom r{k-j}\binom {n+j-1}j
=
(-1)^{s+1}\sum _{t=1} ^{r}\binom {r-t}{k-s}\binom {n-t+s-1}{s-1}.
\end{equation}
In the sum on the left-hand side, we reverse the order of
summation (that is, we replace $j$ by $s-j-1$), and we rewrite
the resulting sum in hypergeometric notation. Thus, we obtain
that the left-hand side of \eqref{eq:10} equals
$$
(-1)^{s+1}\binom r{k-s+1}\binom {n+s-2}{s-1}
{} _{3} F _{2} \!\left [ \begin{matrix} { 1, 1 - s, 1 +
       k - r - s}\\ { 2 - n - s, 2 + k - s}\end{matrix} ; {\displaystyle
       1}\right ].
$$
Next we apply the transformation formula (see \cite[Ex.~7,
p.~98]{BailAA})
$$
{} _{3} F _{2} \!\left [ \begin{matrix} { a, b, c}\\ { d, e}\end{matrix} ;
   {\displaystyle 1}\right ]  =
   \frac {\Gamma (e)\,\Gamma( d + e-a - b - c)}
  {\Gamma(e-a)\,\Gamma(  d + e-b - c)}
  {} _{3} F _{2} \!\left [ \begin{matrix} { a, d-b , d-c}\\ { d, d +
    e-b - c}\end{matrix} ; {\displaystyle 1}\right ]
$$
to the $_3F_2$-series, to obtain
$$
{{\left( -1 \right) }^{s+1}}{\frac {({ \textstyle n}) _{s-1} \,
      ({ \textstyle -k + r + s}) _{1 + k - s} }
    {\left( 1 - n + r \right) \,({ \textstyle 1}) _{k - s} \,
      ({ \textstyle 1}) _{s-1} }}
{} _{3} F _{2} \!\left [ \begin{matrix} { 1 - k -
       n + r, 1, 1 - n}\\ { 2 - n - s, 2 - n + r}\end{matrix} ; {\displaystyle
       1}\right ].
$$
Now we apply the transformation formula
(\cite[Eq.~(3.1.1)]{GaRaAF})
$$
{} _{3} F _{2} \!\left [ \begin{matrix} { a, b, -N}\\ { d, e}\end{matrix} ;
   {\displaystyle 1}\right ]  =
{\frac { ({ \textstyle e-b}) _{N} }
      {({ \textstyle e}) _{N} }}
{} _{3} F _{2} \!\left [ \begin{matrix} { -N, b, d-a}\\ { d, 1 + b - e -
       N}\end{matrix} ; {\displaystyle 1}\right ]
$$
where $N$ is a nonnegative integer. After some simplification,
one sees that the resulting expression agrees with the
right-hand side of \eqref{eq:10}.

\medskip
Now we discuss the algebraic proof. We consider the $u$-th syzygy
$M$ of the residue class ring $S=R/(X_1,\dots,X_r)$ of
$R=K[X_1,\dots,X_n]$. There are two exact sequences from which the
Hilbert function of $M$ can be computed since S is resolved by the
Koszul complex of the sequence $X_1,\dots,X_r$. We simply break the
Koszul complex into two parts, inserting $M$ as the kernel or
cokernel, respectively, at the appropriate place:
\begin{gather*}
0\to M\to\bigwedge^{u-1}F(-u+1)\to\dots\to F\to R\to S\to 0,\qquad F=R^r,\\
0\to\bigwedge^r F(-r)\to\bigwedge^{r-1} F(-r+1) \to\dots\to
\bigwedge^u F(-r) \to M\to 0.
\end{gather*}
The computation of the Euler characteristic of the first complex in
degree $k$ yields the equation
\begin{align}
H(M,k)&=(-1)^u\left(\binom{n-r+k-1}{n-r-1}-
\sum_{i=0}^{u-1}(-1)^i\binom{r}{i}\binom{n+k-i-1}{k-i}\right)\notag\\
&=(-1)^u\left(\binom{n-r+k-1}{n-r-1}-
\sum_{j=k-u+1}^k (-1)^{k-j}\binom{r}{k-j}\binom{n+j-1}{j}
\right),
\label{right}
\end{align}
where we pass from the first to the second line by the substitution
$j=k-i$. In the same degree we obtain for the second complex
\begin{align}
H(M,k)&=\sum_{l=0}^{k-u}(-1)^l\binom{r}{u+l}\binom{n+k-u-l-1}{n-1}\notag\\
&=(-1)^{k-u}\sum_{j=0}^{k-u} (-1)^{j}\binom{r}{k-j}\binom{n+j-1}{j},\label{left}
\end{align}
where  we have substituted the summation index $l$ by $k-u-j$ in the
second line. On the other hand, we can also compute the Hilbert
function by \cite[Proposition~3.7]{BKU}. To this end, we fix $r$.
For $n=r$, \emph{loc.\ cit.} then yields
$$
H(M,k)=\sum_{t=1}^r \binom{r-t}{u-1}\binom{r-t+k-u}{k-u}.
$$
For the ring extension from $K[X_1,\dots,X_r]$ to
$K[X_1,\dots,X_n]$, we have to replace the dimensions of the
symmetric powers of $K^{r-t}$ by those of $K^{n-t}$; therefore
\begin{equation}
H(M,k)=\sum_{t=1}^r \binom{r-t}{u-1}\binom{n-t+k-u}{k-u}\label{hilbsyz}
\end{equation}
in the general case. Setting $s=k-u+1$, one arrives at \eqref{eq:8}
by equating \eqref{right} and \eqref{left}, while equating \eqref{left}
and \eqref{hilbsyz} leads to \eqref{eq:9}.
\end{proof}

\section{The proof of positivity}

In view of Proposition~\ref{prop:0} and Lemma~\ref{lem:4},
for the proof of Theorem~\ref{powers} we have to
show the inequality that we state below in Proposition~\ref{prop:1}.
Its proof requires several auxiliary lemmas, provided for in
Lemmas~\ref{lem:1}--\ref{lem:5}. The actual proof of
Proposition~\ref{prop:1} (and, thus, of Theorem~\ref{powers})
is then given at the end of this section.

\begin{proposition} \label{prop:1}
Let $n$ and $s$ be positive integers, and let $r=\cl{n/(s+1)}$. Then,
for all $k=s+1,s+2,\dots,s+r-1$, we have
\begin{equation} \label{eq:ungl}
\binom {n+k-r-1}k\ge
\sum _{t=1} ^{r}\binom {r-t}{k-s}\binom {n-t+s-1}{s-1}.
\end{equation}
\end{proposition}

\begin{remark} \label{rem:1}
The assertion of the proposition is trivially true if $r-1\le k-s$.
\end{remark}

In the following, we make frequent use of the classical
{\it digamma function} $\psi(x)$, which is  defined to be
the logarithmic derivative of
the {\it gamma function} $\Ga(x)$, i.e., $\psi(x)=\Ga'(x)/\Ga(x)$.

\begin{lemma} \label{lem:1}
For all real {\em(!)} numbers $n,k,s,t$ with $n,t\ge1$, $s\ge2$, $s+2\le
k\le r+s-t-1$, we have
$$
\psi(n+k-r)-\psi(k+1)>\psi(r-t-k+s+1)-\psi(k-s+1),
$$
where, as before, $r=\cl{n/(s+1)}$.
\end{lemma}
\begin{proof}
We want to prove that
\begin{equation} \label{eq:1}
\psi(n+k-r)-\psi(r-t-k+s+1)+\psi(k-s+1)-\psi(k+1)>0
\end{equation}
for the range of parameters indicated in the statement of the lemma.
Since $\psi(x)$ is monotone increasing for $x>0$
(this follows e.g.\ from \cite[Eq.~(1.2.14)]{AnARAA}), the left-hand
side of \eqref{eq:1} is monotone increasing in $t$. It therefore
suffices to prove \eqref{eq:1} for $t=1$, that is, it suffices to prove
\begin{equation} \label{eq:2}
\psi(n+k-r)-\psi(r-k+s)+\psi(k-s+1)-\psi(k+1)>0.
\end{equation}

Next we claim that the left-hand side of \eqref{eq:2} is monotone
increasing in $k$. To see this, we differentiate the left-hand side of
\eqref{eq:2} with respect to $k$, to obtain
\begin{equation} \label{eq:3}
\psi'(n+k-r)+\psi'(r-k+s)+\psi'(k-s+1)-\psi'(k+1).
\end{equation}
Since $\psi(x)$ is monotone increasing, the first two terms in
\eqref{eq:3} are positive. Moreover, $\psi(x)$ is a concave function
for $x>0$ (this follows also from \cite[Eq.~(1.2.14)]{AnARAA}),
whence $\psi'(k-s+1)-\psi'(k+1)>0$. This proves that the expression in
\eqref{eq:3} is positive, that is, that the derivative with respect to
$k$ of the left-hand side of \eqref{eq:2} is positive. This
establishes our claim.

As a result of the above argument, we see that it suffices to prove
\eqref{eq:2} for the smallest $k$, that is, for $k=s+2$. In other
words, it suffices to prove
\begin{equation} \label{eq:4}
\psi(n+s-r+2)-\psi(r-2)+\psi(3)-\psi(s+3)>0.
\end{equation}

We now investigate the behaviour of the left-hand side of \eqref{eq:4}
as a function of $n$, which we denote by $f(n)$ (ignoring the
dependence of the expression on $s$ at this point).
Clearly, as long as $n$ stays strictly between successive
multiples of $s+1$, $r=\cl{n/(s+1)}$ does not change, and $f(n)$
is monotone increasing in $n$ in this
range. However, if $n$ changes from $n=\ell(s+1)$, say, to something
just marginally larger, then $r$ jumps from $\ell$ to $\ell+1$,
thereby changing the value of $f$ discontinuously. The limit value
$\lim_{n\downarrow\ell(s+1)}f(n)$ is given by
$$\lim_{n\downarrow\ell(s+1)}f(n)=
\psi(\ell(s+1)+s-\ell+1)-\psi(\ell-1)+\psi(3)-\psi(s+3).$$
By the argument above, we know that $f(n)$ stays above this value for
$\ell(s+1)<n\le (\ell+1)(s+1)$.
Let us examine the difference of two such limit values:
\begin{align}\notag
\lim_{n\downarrow\ell(s+1)}f(n)&{}
-\lim_{n\downarrow(\ell+1)(s+1)}f(n)=
\psi(\ell(s+1)+s-\ell+1)-\psi(\ell-1)\\
\notag
&\kern4cm
-\psi((\ell+1)(s+1)+s-\ell)+\psi(\ell)\\
&=\psi((\ell+1) s+1)-\psi((\ell+1)s+s+1)
+\frac {1} {\ell-1},
\label{eq:5}
\end{align}
where we used \cite[Eq.~(1.2.15) with $n=1$]{AnARAA} to obtain the
last line. 
By \cite[Eq.~(1.2.12)]{AnARAA}, we have
$\psi'(1)=-\gamma$, where $\gamma$ is the Euler--Mascheroni
constant. Making use of the integral representation
$$\psi(a)=-\gamma+\int _{0} ^{1}\frac {1-x^{a-1}} {1-x}\,dx,\quad
\quad \Re(a-1)>0$$ (see \cite[Theorem~1.6.1(ii) after change of
variables $x=e^{-z}$]{AnARAA}), 
we estimate
\begin{align*}
\psi((\ell+1) s+1)-\psi((\ell+1)s+s+1)
&=
-\int _{0} ^{1}\frac
{x^{(\ell+1) s}-x^{(\ell+1)s+s}} {1-x}\,dx\\
&=-\int _{0} ^{1}x^{(\ell+1) s}\frac
{1-x^{s}} {1-x}\,dx\\
&\ge- s\int _{0} ^{1}x^{(\ell+1)s}\,dx\\
&\ge -\frac s{(\ell+1)s+1}
>-\frac {1} {\ell+1}>-\frac {1} {\ell-1}.
\end{align*}
This shows that the difference in \eqref{eq:5} is (strictly) negative,
that is, that the left-hand side of \eqref{eq:4} becomes smaller when
we ``jump" from (slightly above) $n=\ell(s+1)$ to
(slightly above) $n=(\ell+1)(s+1)$, while the values
in between stay above the limit value from the right
at $n\downarrow(\ell+1)(s+1)$.
Therefore, it suffices to prove \eqref{eq:4} in the limit
$n\to\infty$. By recalling the asymptotic behaviour
\begin{equation} \label{eq:psi}
\psi(x)=\log x+O\left(\frac {1}
{x}\right),
\quad \quad \text {as $x\to\infty$},
\end{equation}
of the digamma function (cf.\ \cite[Cor.~1.4.5]{AnARAA}), we see that this
limit of the left-hand side of \eqref{eq:4} is
$\log s+\psi(3)-\psi(s+3)$, so that it remains to prove
\begin{equation} \label{eq:7}
\log s+\psi(3)-\psi(s+3)>0.
\end{equation}
Also here, we look at the derivative of the left-hand side with
respect to $s$:
$$\frac {1} {s}-\psi'(s+3).$$
By \cite[Eq.~(1.2.14)]{AnARAA}, this can be rewritten in the form
$$\frac {1} {s}-
\sum _{m=0} ^{\infty}\frac {1} {(s+m+3)^2}.$$
The infinite sum can be interpreted as the integral of the step function
$$x\mapsto \frac {1} {\cl{x}^2}$$
between $x=s+2$ and $x=\infty$. The function being bounded above by
the function $x\mapsto 1/x^2$, we conclude
$$\frac {1} {s}-\psi'(s+3)
>\frac {1} {s}-\int _{s+2} ^{\infty}\frac {dx} {x^2}=
\frac {1} {s}-\frac {1} {s+2}>0.$$
In other words, the derivative with respect to $s$ of the left-hand
side of \eqref{eq:7} is always positive, hence it suffices to verify
\eqref{eq:7} for $s=2$:
$$
\log 2+\psi(3)-\psi(5)=\log 2-\frac {1} {3}-\frac {4} {5}>0,
$$
where we used again \cite[Eq.(1.2.15)]{AnARAA}.

This completes the proof of the lemma.
\end{proof}


\begin{lemma} \label{lem:3}
Let the real numbers $n,k_0,s$ be given with $n\ge1$, $s\ge2$, $s+2\le
k_0\le r+s-t-1$. Suppose that \eqref{eq:ungl} holds for this choice of
$n,k_0,s$. Then it also holds for $k$ in an interval $[k_0,k_0+\ep)$
for a suitable $\ep>0$.
\end{lemma}
\begin{proof}
We extend the binomial coefficients in \eqref{eq:ungl} to real values
of $k$, by using gamma functions. To be precise, we extend the
left-hand side of \eqref{eq:ungl} to
$$\frac {\Ga(n+k-r)} {\Ga(k+1)\,(n-r-1)!},$$
and the right-hand side to
$$\sum _{t=1} ^{r}\frac {(r-t)!} {\Ga(k-s+1)\,\Ga(r-t-k+s+1)}
\binom {n-t+s-1}{s-1}.
$$
In abuse of notation, we shall still use binomial coefficient
notation, even if we allow real values of $k$.

We now compute the derivative at $k=k_0$ on both sides of
\eqref{eq:ungl}. On the left-hand side, this is
$$\big(\psi(n+k_0-r)-\psi(k_0+1)\big)\binom {n+k_0-r-1} {k_0},$$
while on the right-hand side this is
$$\sum _{t=1} ^{r}
\big(\psi(r-t-k_0+s+1)-\psi(k_0-s+1)\big)
\binom {r-t}{k_0-s}\binom {n-t+s-1}{s-1}.
$$
By using Lemma~\ref{lem:1}, we can estimate the derivative of the
right-hand side as follows:
\begin{align*}
\sum _{t=1} ^{r}
\big(\psi(r-{}&t-k_0+s+1)-\psi(k_0-s+1)\big)
\binom {r-t}{k_0-s}\binom {n-t+s-1}{s-1}\\
&<\sum _{t=1} ^{r}
\big(\psi(n+k_0-r)-\psi(k_0+1)\big)
\binom {r-t}{k_0-s}\binom {n-t+s-1}{s-1}\\
&<\big(\psi(n+k_0-r)-\psi(k_0+1)\big)
\binom {n+k_0-r-1} {k_0}.
\end{align*}
The last expression is however exactly the derivative of the left-hand
side of \eqref{eq:ungl}. Hence, the left-hand side of \eqref{eq:ungl}
must also exceed the right-hand side in a small ``neighbourhood"
$[k_0,\ep)$ to the right of $k_0$. This proves the lemma.
\end{proof}

\begin{lemma} \label{lem:5}
For all positive integers $n$ and $s$ with $n>3s+3$ and $s\ge2$, we have
\begin{equation} \label{eq:11}
2\binom {n+s-r+1} {s+2}\ge
\binom {n+s+1} {s+2}
-r\binom {n+s} {s+1}
+\binom r2\binom {n+s-1} {s},
\end{equation}
where, as before, $r=\cl{n/(s+1)}$.
\end{lemma}

\begin{proof}
We proceed in a spirit similar to the one in the proof of
Lemma~\ref{lem:3}. We regard both sides of \eqref{eq:11} as
functions in the {\it real\/} variable $n$. The reader should note
that the assumption that $n>3s+3$ implies that $r\ge4$, a fact that
will be used frequently without further mention.

Let first $n$ be strictly between $\ell(s+1)$ and $(\ell+1)(s+1)$,
for some fixed non-negative integer $\ell$. Then $r=\ell+1$, so that
both sides of \eqref{eq:11} become polynomial (whence continuous)
functions in $n$. The derivative of the left-hand side of
\eqref{eq:11} with respect to $n$ (in the interval
$(\ell(s+1),(\ell+1)(s+1))$) equals
\begin{equation} \label{eq:12}
2\big(\psi(n+s-r+2)-\psi(n-r)\big)\binom {n+s-r+1} {s+2},
\end{equation}
while the derivative of the right-hand side of
\eqref{eq:11} with respect to $n$ equals
\begin{equation} \label{eq:13}
\left(\psi(n+s)-\psi(n)+\frac {2n+2s+1-r(s+2)} {N(n,s)}\right)
\frac {(n+s-1)!} {(s+2)!\,(n-1)!}N(n,s),
\end{equation}
where
$$N(n,s)=(n+s)(n+s+1)-r(n+s)(s+2)+\binom r2(s+1)(s+2).$$

We claim that\footnote{Our original proof had a weaker inequality
at this point, which however turns out to be not sufficient.
This gap was pointed out by Jiayuan Lin.
In addition, he provided the following argument establishing 
\eqref{eq:14}, and he kindly gave us the permission to reproduce it
here.}
\begin{equation} \label{eq:14}
\psi(n+s-r+2)-\psi(n-r)>
\psi(n+s)-\psi(n)+\frac {2n+2s+1-r(s+2)} {N(n,s)}.
\end{equation}
This would imply that, provided \eqref{eq:11} holds for
some $n$ in the (closed) interval $[\ell(s+1),(\ell+1)(s+1)]$,
then the derivative of the left-hand side of
\eqref{eq:11} would be larger than the derivative of the right-hand
side of \eqref{eq:11} at this $n$, and hence the function on the
left-hand side of \eqref{eq:11} would grow faster than the right-hand
side of \eqref{eq:11} for $n$ in $(\ell(s+1),(\ell+1)(s+1)]$. In
turn, this would mean that it would suffice to show the validity of
\eqref{eq:11} for $n\downarrow \ell(s+1)$ (that is, for $n=\ell(s+1)$
and $r=\ell+1$), to conclude that \eqref{eq:11} holds for the whole
interval $(\ell(s+1),(\ell+1)(s+1)]$.

We next embark on the proof of \eqref{eq:14}.
Using \cite[Eq.~(1.2.15)]{AnARAA}, we see that
\begin{align*}
\psi(n&+s-r+2)-\psi(n-r)-
\big(\psi(n+s)-\psi(n)\big)
=\sum _{i=0} ^{s+1}\frac {1} {n-r+i}-\sum _{i=0} ^{s-1}\frac {1}
{n+i}\\
&=\frac {1} {n-r}+\frac {1} {n-r+1}+
\sum _{i=0} ^{s-1}\left(\frac {1} {n-r+i+2}-\frac {1} {n+i}\right)\\
&=\frac {2} {n-r+1}+\frac {1} {(n-r)(n-r+1)}+
\sum _{i=0} ^{s-1}\frac {r-2} {(n-r+i+2)(n+i)}\\
&>\frac {2} {n-r+1}+\frac {1} {(n-r)(n-r+1)}+
\sum _{i=0} ^{s-1}\frac {r-2} {(n+i-1)(n+i)}\\
&>\frac {2} {n-r+1}+\frac {1} {(n-r)(n-r+1)}+
(r-2)\sum _{i=0} ^{s-1}\left(\frac {1} {n+i-1}-\frac {1}
{n+i}\right)\\
&>\frac {2} {n-r+1}+\frac {1} {(n-r)(n-r+1)}+
(r-2)\left(\frac {1} {n-1}-\frac {1}
{n+s-1}\right)\\
&>\frac {2} {n-r+1}+\frac {1} {(n-r)(n-r+1)}+
\frac {(r-2)s} {(n-1)(n+s-1)}\\
&>\frac {2} {n-r+1}+\frac {(r-2)s+1} {(n-1)(n+s-1)}.
\end{align*}
Hence, if we are able to prove that
\begin{equation} \label{eq:14A}
\frac {2} {n-r+1}+\frac {(r-2)s+1} {(n-1)(n+s-1)}\ge
\frac {2n+2s+1-r(s+2)} {N(n,s)},
\end{equation}
the inequality \eqref{eq:14} will follow immediately. 
We now claim that
\begin{equation} \label{eq:15A}
(r-1)N(n,s)\ge (n-1)(n+s)
\end{equation}
and
\begin{equation} \label{eq:15B}
2N(n,s)+((r-2)s+1)s\ge(n-r+1)(2n+2s+1-r(s+2)).
\end{equation}
If, for the moment, we assume the validity of \eqref{eq:15A} and
\eqref{eq:15B}, then we infer
\begin{align*}
\frac {2} {n-r+1}+\frac {(r-2)s+1} {(n-1)(n+s-1)}
&=\frac {1} {n-r+1}\left(2+\frac {((r-2)s+1)(n-r+1)}
{(n-1)(n+s-1)}\right)\\
&\ge\frac {1} {n-r+1}\left(2+\frac {((r-2)s+1)(n-r+1)}
{(r-1)N(n,s)}\right)\\
&\ge\frac {1} {n-r+1}\left(2+\frac {((r-2)s+1)s}
{N(n,s)}\right)\\
&\ge\frac {2n+2s+1-r(s+2)} {N(n,s)},
\end{align*}
which is exactly \eqref{eq:14A}. Here we used \eqref{eq:15A} to
obtain the second line, the simple fact that 
$$\frac {n-r+1} {r-1}\ge\frac {(r-1)(s+1)-r+1} {r-1}=s$$
to obtain the third line, and \eqref{eq:15B} to obtain the last
line. In summary, \eqref{eq:15A} and \eqref{eq:15B} together would
imply \eqref{eq:14A}, and hence \eqref{eq:14}.

To see \eqref{eq:15A}, we rewrite it explicitly in the form
\begin{equation} \label{eq:15C}
(r-1)\binom r2(s+1)(s+2)
\ge (n+s)\big(n-1-(r-1)(n+s+1)+r(r-1)(s+2)\big).
\end{equation}
We write $n=r(s+1)-n_0$, with $0\le n_0\le s$. If we substitute this
in the inequality above, then the right-hand side of \eqref{eq:15C} 
turns into
$$
\big(r(s+1)+s-n_0\big)
\big(r(s+1)-1-(r-1)(r+1)(s+1)+(r-2)n_0+r(r-1)(s+2)\big).
$$
We consider this as a quadratic function in $n_0$. It has its unique 
maximum at
$$
\frac {r^2s-rs-r-3s} {2(r-2)}\ge
\frac {3rs-r-3s} {2(r-2)}\ge
\frac {2rs-4s} {2(r-2)}=s.
$$
It is therefore monotone increasing on the interval $[0,s]$ and
consequently attains its maximal value on the interval $[0,s]$ at $n_0=s$.
So it suffices to verify \eqref{eq:15C} at $n_0=s$. After
simplification, this turns out to be equivalent to 
$$\binom r2(s+1)(s(r-3)-2)\ge0,$$
which, by our assumptions on $r$ and $s$, is trivially true.

To see \eqref{eq:15B}, we again substitute $r(s+1)-n_0$ for $n$
(with $0\le n_0\le s$) to obtain the equivalent inequality
$$
rs(r-3)+s-1+(rs-2s+1)n_0>0,
$$
which is trivially true by our assumptions on $r$, $s$, and $n_0$.

Altogether, we have now established \eqref{eq:14}.
Hence, the conclusion of the paragraph following \eqref{eq:14}
that it suffices to
prove \eqref{eq:11} for $n=\ell(s+1)$ and $r=\ell+1$ holds as well.

We substitute $n=\ell(s+1)$ and $r=\ell+1$ in \eqref{eq:11}:
\begin{multline*} 
2\binom {(\ell+1)s} {s+2}\ge
\binom {(\ell+1)(s+1)} {s+2}\\
-(\ell+1)\binom {(\ell+1)(s+1)-1} {s+1}
+\binom {\ell+1}2\binom {(\ell+1)(s+1)-2} {s},
\end{multline*}
and, after simplification, we obtain the equivalent inequality
\begin{equation*} 
2\frac {((\ell+1)s-1)!} {(\ell s-2)!}\ge
\frac {1} {2}\frac {((\ell+1)(s+1)-2)!\,(\ell(s+1)-2)} {(\ell(s+1)-1)!}.
\end{equation*}

We shall actually establish the stronger inequality
\begin{equation} \label{eq:19}
2\frac {((\ell+1)s-1)!} {(\ell s-2)!}\ge
\frac {1} {2}\frac {((\ell+1)(s+1)-2)!} {(\ell(s+1)-2)!}.
\end{equation}
In order to do so, we regard the functions in \eqref{eq:19} again as
functions in {\it real\/} variables, more precisely, as functions
in the {\it real\/} variable $\ell$, while we think of $s$ as being
fixed.

We first verify that \eqref{eq:19} holds for the
smallest possible value of $\ell=r-1$, that is, for $\ell=3$. For
that value of $\ell$, the inequality \eqref{eq:19} becomes
$$
2\frac {(4s-1)!} {(3 s-2)!}\ge
\frac {1} {2}\frac {(4s+2)!} {(3s+1)!},
$$
or, equivalently,
$$
4(3s+1)3s(3s-1)\ge
(4s+2)(4s+1)4s,
$$
which is indeed true for $s\ge2$.

Next we compute the derivative of both sides of \eqref{eq:19} with
respect to $\ell$. On the left-hand side, we obtain
\begin{equation} \label{eq:20}
2s\left(\psi((\ell+1)s)-\psi(\ell s-1)\right)
\frac {((\ell+1)s-1)!} {(\ell s-2)!},
\end{equation}
while on the right-hand side we obtain
\begin{equation} \label{eq:21}
\frac {s+1} {2}
\left(\psi((\ell+1)(s+1)-1)-\psi(\ell(s+1)-1)\right)
\frac {((\ell+1)(s+1)-2)!} {(\ell(s+1)-2)!}.
\end{equation}
Using \cite[Eq.~(1.2.15)]{AnARAA}, it is straightforward to see that
$$\psi((\ell+1)s)-\psi(\ell s-1)>
\psi((\ell+1)(s+1)-1)-\psi(\ell(s+1)-1).$$
Furthermore, we have $2s>\frac {s+1} {2}$, so that
$$
2s\left(\psi((\ell+1)s)-\psi(\ell s-1)\right)>
\frac {s+1} {2}
\left(\psi((\ell+1)(s+1)-1)-\psi(\ell(s+1)-1)\right)
$$
for all $\ell\ge3$. Since we already know that \eqref{eq:19} holds
for $\ell=3$, it then follows that the derivative of the left-hand
side of \eqref{eq:19} (see \eqref{eq:20}) is always larger than the
derivative of the right-hand side (see \eqref{eq:21}). This
establishes \eqref{eq:19} and completes the proof of the lemma.
\end{proof}

\begin{proof}[Proof of Proposition~\ref{prop:1}]
The assertion is true for $k=s+1$ because of the choice of $r$.
By comparing \eqref{eq:8} and \eqref{eq:9} in
Lemma~\ref{lem:4}, the assertion for $k=s+2$, which reads
$$
\binom {n+s-r+1}{s+2}\ge
\sum _{t=1} ^{r}\binom {r-t}{2}\binom {n-t+s-1}{s-1},
$$
can be rewritten as
$$
\binom {n+s-r+1}{s+2}\ge
- \binom {n+s-r+1}{s+2}+
\sum _{j=s} ^{s+2}(-1)^{s-j+2}\binom {n+j-1}j\binom r{s-j+2},
$$
or, equivalently, as
$$
2\binom {n+s-r+1} {s+2}\ge
\binom {n+s+1} {s+2}
-r\binom {n+s} {s}
+\binom r2\binom {n+s-1} {s}.
$$
Remembering Remark~\ref{rem:1}, we see that it is enough to show this
for $r>3$, that is, for $n>3s+3$.
Lemma~\ref{lem:5} shows that the above inequality indeed holds for
that range of $n$.
Lemma~\ref{lem:3} then implies that the assertion must be true for all
$k\ge s+2$.
\end{proof}

\end{document}